\documentclass{amsart}
\usepackage{amssymb,amsthm,amscd}
\usepackage[pdftex]{graphicx}
\include{diagrams}
\usepackage[pagebackref]{hyperref}

\newcommand{\rsp}{\raisebox{0em}[2ex][1.3ex]{\rule{0em}{2ex} }}
\newcommand{\smatr}[4]{(\begin{smallmatrix} 
             #1 & #2 \\ #3 & #4 \end{smallmatrix})}

\newcommand{\N}{{\mathbb N}}
\newcommand{\Z}{{\mathbb Z}}
\newcommand{\bP}{{\mathbb P}}
\newcommand{\Q}{{\mathbb Q}}

\newcommand{\cB}{{\mathcal B}}
\newcommand{\cO}{{\mathcal O}}

\newcommand{\eps}{\varepsilon}
\newcommand{\fa}{{\mathfrak a}}
\newcommand{\fb}{{\mathfrak b}}

\newcommand{\bbeta}{\overline{\beta}}

\newcommand{\bmu}{\overline{\mu}}

\newcommand{\bpi}{\overline{\pi}}
\newcommand{\brho}{\overline{\rho}}
\newcommand{\SL}{{\operatorname{SL}}}
\newcommand{\Cl}{{\operatorname{Cl}}}
\newcommand{\Gal}{{\operatorname{Gal}}}
\newcommand{\disc}{{\operatorname{disc}}}
\newcommand{\Tr}{{\operatorname{Tr}}}
\newcommand{\lra}{\longrightarrow}
\newcommand{\Lra}{\Longrightarrow}

\newcommand{\gen}{{\operatorname{gen}}}
\newcommand{\hphi}{\widehat{\phi}}
\newcommand{\by}{\overline{y}}
\newcommand{\bT}{\overline{T}}
\newcommand{\tT}{\widetilde{T}}
\newcommand{\la}{\langle}
\newcommand{\ra}{\rangle}
\newcommand{\ov}{\overline}

\newtheorem{thm}{Theorem}
\newtheorem{prop}[thm]{Proposition}
\newtheorem{lem}[thm]{Lemma}
\newtheorem{cor}[thm]{Corollary}

\title{Hilbert $2$-Class Fields and $2$-Descent}
\author{F. Lemmermeyer}
\begin{document}

\begin{abstract}
We give a construction of unramified cyclic octic extensions of 
certain complex quadratic number fields. The binary quadratic
form used in this construction also shows up in the theory of
$2$-descents on Pell conics and elliptic curves, as well as in 
the explicit description of cyclic quartic extensions.
\end{abstract}

\maketitle

\begin{center} \today \end{center}

\section*{Introduction}

In this article we discuss several apparently unrelated problems 
involving integers of the form $m = a^2 + 4b^2 \equiv 1 \bmod 4$. 
A central role in the solution of these problems is played by the 
family of binary quadratic forms $Q_b = (b,a,-b)$ with discriminant $m$;
here and below, $Q = (A,B,C)$ denotes the binary quadratic form
$Q(X,Y) = AX^2 + BXY + CY^2$ with discriminant $\Delta = B^2 - 4AC$.

The first problem concerns the solvability of the negative Pell 
equation 
$$ T^2- mU^2 = -4 $$ 
and the computation of its fundamental solution. Results going back 
to Euler \cite{Euler} show that such a solution can be computed from 
an integral solution of the equation $Q_b(r,s) = 1$.

The second problem is the explicit construction of octic 
cyclic unramified  extensions $L/k$ of the quadratic number 
field $k = \Q(\sqrt{-m}\,)$. The construction of its quartic 
subextension $K = k(i,\sqrt{a+2bi}\,)$ is almost trivial, and 
we will see that explicit generators  $\mu$ of the quadratic 
extension $L = K(\sqrt{\mu}\,)$ can be written down explicitly 
using nontrivial solutions of the diophantine equation 
$Q_b(r,s) = 2x^2$.

Afterwards we will briefly explain why the forms $Q_b$ also play a
role in performing $2$-descent on certain families of elliptic curves, 
such as those of the form $E: y^2 = x(x^2 + p)$, where $p = a^2 + 4b^2$ 
is prime; here the existence of rational points on $E$ is tied to 
the pair of equations $r^2 + s^2 = X^2$ and $Q_b(r,s) = Y^2$.

Finally we mention a few other problems in which these 
forms $Q_b$ have shown up. It is clear that most of the
problems discussed here may be generalized considerably. 
In particular, studying Pell descent or Hilbert $2$-class fields
should by no means be restricted to the special cases of the 
negative Pell equation or discriminants of the form $-4m$. 

This article is written in the language of quadratic forms
(although ideals show up occasionally). For understanding the
results it is sufficient to know the most elementary basics;
for readers who would like to read more about reduction and composition 
of forms, I strongly recommend the books by Flath \cite{Fla} and 
Cox \cite{Cox}, as well as the recent contributions by Bhargava 
\cite{Bha}. The first section of \cite{LPep} gives a brief 
introduction to Bhargava's ideas, and a detailed elementary 
account can be found in Barker-Hoyt's thesis \cite{BH}.

\section{The Quadratic Space of Binary Quadratic Forms}
Before we start investigating examples of descent on Pell conics,
let us show how the main actor in our play shows up in a more
general setting.

Let $R$ be a domain with characteristic $\ne 2$, and consider the set 
$\cB = \cB_R$ of binary quadratic forms $Q(X,Y) = AX^2 + BXY + CY^2$, 
often abbreviated by $Q = (A,B,C)$, with $A, B, C \in R$. 
We define a bilinear map $\cB \times \cB \lra R$ via 
$$ \la Q_1 , Q_2 \ra = B_1B_2 - 2A_1B_2 - 2A_2B_1, $$ 
where $Q_j = (A_j,B_j,C_j) \in \cB$. Clearly 
$\la Q_1 , Q_2 \ra = \la Q_2 , Q_1 \ra$, and 
$\la Q , Q \ra= \disc(Q) = B^2 - 4AC$ is the discriminant of $Q$.

Matrices $S = \smatr{r}{s}{t}{u} \in \SL_2(R)$ act on $\cB$ via 
$Q|_S = Q'$, where
\begin{equation}\label{EQS}
  Q'(x,y) = Q \big((x,y)S'\big)  = Q(rx+sy,tx+uy),
\end{equation}
(here $S'$ denotes the transpose of $S$). It is an easy exercise to 
show that 
\begin{equation}\label{Eactb}
              \la {Q_1}|_S, {Q_2} \ra = \la Q_1, Q_2|_{S'} \ra
\end{equation} 
for $S \in \SL_2(\Q)$. This identity is the essential content of the
``Cantor diagrams'' in  \cite[Thm. 2]{Pall}, \cite[Prop. 3.3]{Brw1},
and \cite[Thm. 2]{Brw2}.

From now on assume that $R = \Z$. If the form $Q_1 = (1,0,-m)$ 
represents $-1$, say $r^2 - ms^2 = -1$, then $Q_2 = (ms,2r,s)$ 
is a form with discriminant $-4$. Therefore $Q_2$ is equivalent
to $(1,0,1)$, say $Q_2 = (1,0,1)|_S$ for some $S \in \SL_2(\Z)$.
Write $Q_1|_{S'} = (a,2b_1,c)$ for integers $a, b_1, c$; then 
Equation (\ref{Eactb}) shows that $a+c = 0$, i.e., 
$Q_1|_{S'} = (a,2b_1,-a)$. Since $Q_1$ and $Q_1|_{S'}$ both have
discriminant $4m$, we must have $m = a^2 + b_1^2$. Observe that we 
have shown that if $(1,0,-m)$ represents $-1$, then $m = a^2 + b_1^2$ 
is a sum of two squares.

Replacing the principal form $(1,0,-a)$ with discriminant $4a$ by 
$Q_0 = (1,1,\frac{1-a}4)$ with discriminant 
$a = a^2 + 4b^2 \equiv 1 \bmod 4$ results in replacing $(a,2b_1,-a)$ 
by $Q_b = (b,a,-b)$. This is the form that will play a central role
in the problems described below: performing a second descent on Pell 
conics and elliptic curves, constructing Hilbert class fields of 
quadratic number fields with discriminant $-4m$, or in Hasse's 
description of the arithmetic of cyclic quartic number fields.

The pair of orthogonal forms $Q_1 = (1,0,1)$ and $Q_2 = (b,a,-b)$ will 
occur explicitly in our description of the second $2$-descent on the
elliptic curve $y^2 = x(x^2 - 4p)$ for primes $p = a^2 + 4b^2$.
In fact we will see that finding a rational point on $E$ is equivalent
to finding a simultaneous representation $Q_1(r,s) = X^2$ and 
$Q_2(r,s) = Y^2$ of squares by these forms. 

This problem can be reformulated as follows. Let $Q = (A,B,C) \in \cB_R$ 
be a binary quadratic form over a domain $R$ as above, and let $F$ be 
the quotient field of $R$. We can evaluate $Q$ on $\bP^1 F$ as follows: 
for each point $P = [x:y] \in \bP^1 F$ we set $P = Q(x,y)$ with values in 
$F^\times/F^{\times\,2}$. Given a pair $(Q_1,Q_2)$ of orthogonal forms, the 
pair of simultaneous equations $Q_1(r,s) = X^2$ and $Q_2(r,s) = Y^2$ is 
now equivalent to the existence of a point $P \in \bP^1 \Q$ with 
$Q_1(P) = Q_2(P) = 1$.

\section{ $2$-Descent on Pell Conics}\label{SN1}
Let $m = 4n+1$ be a squarefree integer, and consider 
the Pell equation
\begin{equation}\label{Epelln}
   Q_0(T,U) = T^2 - TU  - nU^2 = 1,
\end{equation}
where $Q_0(X,Y) = X^2 - XY - nY^2$ is the principal binary quadratic form
with discriminant $4n+1 = m$. 

Multiplying through by $4$ and completing the square shows that this
equation can also be written in the more familiar way
$$   (2T-U)^2 - m U^2 = 4. $$

\subsection{First $2$-Descent}
In this section we will present methods for finding a nontrivial
integral solution of (\ref{Epelln}) as well as criteria for the
solvability of the negative Pell equation 
\begin{equation}\label{E1Dn} 
      Q_0(T,U) = -1.
\end{equation}

To this end we write the Pell equation $ (2T-U)^2 - m U^2 = 4$ in the form
$m U^2 = (2T-U)^2 - 4 = (2T-U-2)(2T-U+2)$. 

Now $$ \gcd(2T-U-2,2T-U+2) = \begin{cases}
                           4 & \text{ if $U$ is even}, \\ 
                           1 & \text{ if $U$ is odd}; \\ 
       \end{cases} $$
in both cases we find that there exist integers $c, d, r, s$ with 
$cd = m$ and $2T-U+2 = cr^2$ and $2T-U-2 = ds^2$, and so
$cr^2 - ds^2 = 4$.

The fact that $r \equiv s \bmod 2$ allows us to make the
substitution $s = S$ and $r = 2R-S$, which transforms the 
equation $cr^2 - ds^2 = 4$ into 
\begin{equation}\label{Edesc1}
     cR^2 - cRS + \frac{c-d}4 S^2 = 1. 
\end{equation}     
Observe that the quadratic form $cx^2 - dy^2$ has discriminant $4cd = 4m$, 
whereas the form on the left hand side of (\ref{Edesc1}) has discriminant 
$c^2 - c(c-d) = cd = m$.

Thus each integral solution of the Pell equation comes from a solution of 
one of the equations (\ref{Edesc1}) for some factorization $cd = m$ (see 
\cite{LPd} for details). In fact, Legendre claimed and Dirichlet proved 
that exactly four among these equations\footnote{If $m = 5$, for example, 
the solvable equations correspond to the factorizations 
$(c,d) = (1,5)$, $(5,1)$, $(-1,-5)$, and $(-5,-1)$. If 
$m = 34$, on the other hand, the solvable equations come from
$(c,d) = (1,34)$, $(-34,-1)$, $(2,17)$ and $(-17,-2)$.}
have solutions in integers:

\begin{prop}
First $2$-descent on Pell conics: If $m = a^2 + 4b^2 = p_1 \cdots p_t$ 
is a product of $t$ disctinct primes $p_j \equiv 1 \bmod 4$, then among the 
$2^{t+1}$ descendants (\ref{Edesc1}) with $cd = m$ there are exactly 
four that have solutions in integers.
\end{prop}

In particular, apart from the trivial descendants $1 = x^2 - m y^2$ and 
$1 = -m x^2 + y^2$ there is a unique nontrivial pair $(c,d)$ with
$cd = m $ such that $1 = cx^2 - dy^2$ and $1 = -dx^2 + cy^2$
have integral solutions.

\subsection*{Special Cases}
If $m  = p$ is prime, all four factorizations of $m $ must lead
to solvable equations. In particular, $T^2 - pU^2 = -1$ is solvable
in integers. This result is due to Legendre.

If $m  = pq$ for primes $p \equiv  q \equiv 1 \bmod 4$, we have
the following essentially different descendants, written using forms 
with discriminant $4pq$:
$$ x^2 - pqy^2 = 4, \quad pqx^2 - y^2 = 4, \quad 
   px^2 - qy^2 = 4, \quad qx^2 - py^2 = 4. $$
If $(\frac pq) = -1$, it is immediately clear that the last two 
equations do not have rational (let alone integral) solutions; this implies

\begin{cor}
If $p \equiv q \equiv 1 \bmod 4$ are primes with $(\frac pq) = -1$, the
negative Pell equation (\ref{E1Dn}) is solvable in integers.
\end{cor}

If $(\frac pq) = 1$, on the other hand, further descents are necessary for 
deciding which of the equations (\ref{Edesc1}) are solvable.

\subsection{Second $2$-Descent}\label{S1.2}

Let us show how to do such a second descent. We start with the equation
$T^2 - m U^2 = -4$ and assume that it has an integral solution. Then 
$m U^2 = T^2 + 4 = (T+2i)(T-2i)$.

Now $T$ is easily seen to be either odd or divisible by $4$. Thus
$\gcd(T+2i,T-2i) = 1$ in the first and $\gcd(T+2i,T-2i) = 2i = (1+i)^2$
in the second case. This shows that we must have 
$\mu \rho^2 = j(T+2i)$ for some unit $j = i^k$ and some 
$\mu, \rho \in \Z[i]$ with $\mu \bmu = m $ and 
$\rho\brho = u$. Subsuming $j$ into $\mu$ then gives 
$\mu \rho^2 = T+2i$. 

A simple calculation shows that $\mu \equiv 1 \bmod 2$; 
writing $\mu = a + 2bi$ (recall that $m  = \mu \bmu = a^2 + 4b^2$)
and $\rho = x+yi$, and comparing real and imaginary parts we find that
$$ T = ax^2 - 4bxy - ay^2, \qquad 1 = bx^2 + axy - by^2. $$ 

We have proved the following  

\begin{prop}\label{PrQan}
Let $m $ be an odd squarefree natural number. The negative Pell equation 
(\ref{E1Dn}) has an integral solution if and only if there exist integers 
$a, b \in \N$ with $m  = a^2 + 4b^2$ such that the diophantine 
equation
\begin{equation}\label{E2Dn}
   Q_b(x,y) =  bx^2 + axy - by^2 = 1 
\end{equation}
has an integral solution. In this case, 
$$ T = ax^2 - 4bxy - ay^2, \quad U = x^2 + y^2 $$
is a solution of the original equation $Q_0(T) = -1$.
\end{prop}

\medskip \noindent{\bf Example 1.}
Let $m = 41 = 5^2 + 4 \cdot 2^2$; for checking the solvability of 
$T^2 - TU - 10U^2 = -1$ we have to consider $Q(r,s) = 2x^2 + 5xy - 2y^2 = 1$.
We find the solution $x = 3$, $y = -1$, hence $T = 5\cdot 3^2 + 8 \cdot 3 - 5 = 64$  
and $U = 3^2 + 1^2 = 10$. Thus $(T,U) = (64,10)$ is the fundamental solution
of the negative Pell equation, and in fact we have
$64^2 - 41 \cdot 10^2 = -4$, or $32^2 - 41 \cdot 5^2 = -1$ and
$37^2 - 37 \cdot 10 - 10 \cdot 10^2 = -1$.

\medskip \noindent{\bf Example 2.}
Let $m = 221$. Then $m = 5^2 + 14^2 = 11^2 + 10^2$, so we have to 
solve the equations
$$  7x^2 +  5xy - 7y^2 = 1  \quad \text{or} \quad
    5x^2 + 11xy - 5y^2 = 1. $$   
These equations can be written in the form
$$ (14x + 5y)^2 - m y^2 = 28 \quad \text{and} \quad
   (10x + 11y)^2 - m y^2 = 44.  $$ 
Neither of these equations has a solution since 
$(\frac7{13}) = (\frac{11}{13}) = -1$; thus the negative 
Pell equation $T^2 - 221U^2 = -4$ is not solvable.

\medskip \noindent{\bf Example 3.}
If $m = 4777 = 17 \cdot 281 = 59^2 + 4 \cdot 18^2 = 69^2 + 4 \cdot 2^2$, 
then we have to check the solvability of the equations
$$  18x^2 + 59xy - 18y^2 = 1, \quad \text{and} \quad 
     2x^2 + 69xy - 2y^2  = 1. $$
For solving the first equation, observe that $\xi = \frac xy$
is approximately equal to one of the roots of the equation
$18\xi^2  + 59\xi - 18 = 0$, that is, to $\xi \approx 3.56$ or to 
$\xi \approx 0.28$. Using these approximations it is easy to solve 
the first equation by a brute force computation: we find
$(x,y) = (587,2089)$, hence 
$t = \frac12(- 59x^2 + 72xy + 59y^2) = 162715632$ 
and $u = \frac12(x^2+y^2) = 2354245$. 

The second form $Q = (2,69,-2)$ is not equivalent to the principal form
and generates a class of order $2$. In fact, we have 
$Q(587,-17) = 9$, so $Q \sim Q_1^2$ for some form $Q_1 = (3,*,*)$, 
which is not contained in the principal genus, hence cannot be
equivalent to the principal form. Equivalently, the ambiguous
form $Q_{17} = (17,17,-66)$ generates a class of order $2$, and
Gauss composition shows $Q \cdot Q_{17} \sim Q_2 = (2,69,-2)$.

\subsection*{Rational Solvability}
A necessary condition for the solvability of a diophantine equation
in integers is its solvability in rational numbers. Applied to the 
negative Pell equation (\ref{E1Dn}), this gives us the classical 
observation that $(\frac{-1}p) = +1$ for all primes $p \mid m$, 
which is equivalent to $m$ being a sum of two squares. 

Although the solvability of (\ref{E1Dn}) in integers is equivalent to the
solvability of one of the equations (\ref{E2Dn}) in integers, we get stronger
conditions by applying the above observation to (\ref{E2Dn}). In fact, the 
solvability of (\ref{E2Dn}) in rational numbers is easy to check: 

\begin{lem}
Equation (\ref{E2Dn}) has a solution in rational numbers if and only
if $(\frac ap) = +1$ (or, equivalently, $(\frac bp) = +1$) for all 
primes $p \mid m$.
\end{lem}

\begin{proof}
Multiplying (\ref{E2Dn}) through by $4b$ and completing
the square gives $(2bx + ay)^2 - m y^2 = 4b$. The fact that
$(\frac ap) = (\frac bp)$ follows immediately from the 
congruence $(a+2b)^2 \equiv 4ab \bmod p$.

We have to check solvability in all completions of $\Q$. Solvability 
in the reals being clear, we have to verify local solvability in $\Q_p$ 
for all primes $p$. This is equivalent to the triviality of the Hilbert 
symbol $\big(\frac{b,m}p\big) = 1$ for all primes $p$. By the product 
formula $\prod_p \big(\frac{b,m}p\big) = 1$, it is enough to check 
solvability for all primes except one. We will therefore show that 
the equation above has solutions in the completions $\Q_p$ for all 
odd primes $p$. By Hensel's Lemma it is sufficient to check solvability 
modulo $p$. Now we distinguish two cases:

\medskip\noindent
$p \nmid m$: Then $X^2 - m Y^2$ represents at least one nonsquare mod $p$,
hence all of them; thus $X^2 - m Y^2 \equiv 4b \bmod p$ is
solvable, and since $\gcd(a,2b) = 1$, we can write 
$Y = y$ and $X = 2bx + ay$.

\medskip\noindent
$p \mid m$:  Then $(2bx + ay)^2 - m y^2 \equiv  4b \bmod p$ implies 
$(\frac bp) = +1$. If conversely $(\frac bp) = +1$, we can show
solvability modulo $p$ and in $\Z_p$ exactly as above. 
\end{proof}

By Gauss's genus theory, the conditions $(\frac bp) = +1$ for all
primes $p \mid m$ is equivalent to $Q_b$ being in the principal genus.
Such forms are known to represent $1$ rationally. Solvability criteria
for the negative Pell equation following from this criterion are due
to Dirichlet, Scholz, Epstein and others.

\subsection*{Special Cases}
Assume that $m = pq$ is the product of two primes $p \equiv q \equiv 1 \bmod 4$.
We have already seen that (\ref{E1Dn}) is solvable in integers if $(\frac pq) = -1$.
Suppose therefore that $(\frac pq) = -1$, and write $p = A^2 + 4B^2$ and 
$q = C^2 + 4D^2$. Then $m = a^2 + 4b^2$ for $(a,b) = (AC-4BD,AD+BC)$ and 
$(a,b) = (AC+4BD,AD-BC)$. By Burde's reciprocity law (see \cite{LRL}) we have
$$ \Big(\frac{AC \pm 4BD}p \Big) = \Big(\frac{AC \pm 4BD}q \Big) = 
    \Big(\frac pq \Big)^{\phantom{p}}_4  \Big(\frac qp \Big)^{\phantom{p}}_4. $$

Thus if $(\frac pq)^{\phantom{p}}_4 = - (\frac qp)_4$, the second 
descendants $Q_b(x,y) = 1$ are not solvable, hence neither is the 
negative Pell equation (\ref{E1Dn}).

If  $(\frac pq)^{\phantom{p}}_4 = (\frac qp)_4 = -1$, on the other hand, the 
first descendants $px^2 - qy^2 = \pm 1$ do not have solutions. In fact from 
$px^2 - qy^2 = 1$ we get 
$(\frac qp)^{\phantom{p}}_4 = (\frac yp)(\frac{-1}p)^{\phantom{p}}_4$, 
and with $y = 2^ju$ for some $j \ge 1$ we find 
$(\frac yp) = (\frac 2p)^j (\frac pu) = (\frac 2p) 
            = (\frac{-1}p)^{\phantom{p}}_4$.
Thus the solvability of $px^2 - qy^2 = 1$ implies 
$(\frac qp)^{\phantom{p}}_4 = 1$. Similarly, $qx^2 - py^2 = 1$ 
can only be solvable if $(\frac pq)^{\phantom{p}}_4 = 1$. Thus if 
$(\frac pq)^{\phantom{p}}_4 = (\frac qp)^{\phantom{p}}_4 = -1$, 
neither of the equations $px^2 - qy^2 = \pm 1$ is solvable; but 
then one of the equations $Q_b(x,y) = 1$ must have a solution.
We have proved

\begin{cor}
Let $m = pq$ be a product of two primes $p \equiv q \equiv 1 \bmod 4$.
Then we have the following possibilities:
\begin{enumerate}
\item $(\frac pq) = -1$: then (\ref{E1Dn}) is solvable.
\item $(\frac pq) = +1$: If 
      $(\frac pq)^{\phantom{p}}_4 = - (\frac qp)^{\phantom{p}}_4$, 
      then (\ref{E1Dn}) is not solvable. \\
\phantom{$(\frac pq) = +1$:} If 
      $(\frac pq)^{\phantom{p}}_4 = (\frac qp)^{\phantom{p}}_4 = -1$, then 
      (\ref{E1Dn}) is solvable. 
\end{enumerate}   
\end{cor}

These results go back to Dirichlet and Scholz. It is not very hard to
produce numerous similar results for products of three and more primes.

\subsection*{Classes represented by the forms $Q_b$}

The propositions above show that if the negative Pell equation 
$x^2 - my^2 = -1$ is solvable, then one of the forms $Q_b$ 
(where $b$ runs through all integers with $m = a^2 + 4b^2$) 
represents $1$, hence lies in the principal class of the 
primitive binary quadratic forms with discriminant $m$. Our 
next result tells us exactly how the remaining forms $Q_b$ are 
distributed among the equivalence classes of order dividing $2$, 
and that the form $Q_b$ representing $1$ is unique: 

\begin{thm}\label{Tquf}
Let $m = p_1 \cdots p_t$ be a product of primes $p_j \equiv 1 \bmod 4$,
and let $m = a_j^2 + 4b_j^2$ ($1 \le j \le t$, $a_j, b_j > 0$) be the 
different representations of $m$ as a sum of two squares. Then the
binary quadratic forms $Q_j = (b_j,a_j,-b_j)$ have discriminant $m$,
and there are two cases:
\begin{enumerate}
\item The negative Pell equation is solvable: then the forms $Q_b$
      represent the $2^{t-1}$ classes of order dividing $2$ in $\Cl(m)$,
      the class group of binary quadratic forms with discriminant $m$.
      In particular, the form representing $1$ is unique.
\item The negative Pell equation is not solvable: then there exists a 
      subgroup $C$ of
      index $2$ in $\Cl(m)[2]$ such that each form $Q_b$ represents one class
      in $\Cl(m)[2] \setminus C$. Each such class is represented
      by exactly two forms.
\end{enumerate}
In the first case, the pair $(a,b)$ for which $Q_b = (b,a,-b)$ represents $1$
is unique.
\end{thm}

Under the standard correspondence between classes of binary quadratic
forms and ideal classes, the class of the form $Q_b = (b,a,-b)$ with 
discriminant $a^2 + 4b^2 = m$ corresponds to the ideal class
represented by the ideal $(\frac{a+\sqrt{m}}2, b)$.

If $m = p_1 \cdots p_t$ is the product of $t$ prime factors
$p_j \equiv 1 \bmod 4$, then $m = a_j^2 + 4b_j^2$ can be written in 
$2^{t-1}$ ways as a sum of two squares of positive integers. To each such 
sum we attach an ideal
$$ \fa_j = (2b_j + \sqrt{m}, a_j), $$ 
which has the properties
$$ N\fa_j = a_j, \qquad \fa_j^2 = (2b+\sqrt{m}\,). $$

In \cite{Lqu} we have proved the following theorem, which is the ideal
theoretic version of Thm. \ref{Tquf}:

\begin{thm}\label{Th1}
Let $K = \Q(\sqrt{m}\,)$ be a quadratic number field, where
$m = p_1 \cdots p_t$ is a product of primes $p_j \equiv 1 \bmod 4$. 
Let $\eps$ denote the fundamental unit of $K$.
\begin{itemize}
\item If $N \eps = -1$, then the ideal classes $[\fa_j]$ are pairwise distinct and
      represent the $2^{t-1}$ classes of order dividing $2$ in $\Cl(K)$. Each ideal
      $\fa_j$ is equivalent to a unique ramified ideal $\fb_e$. In particular, exactly
      one of the $\fa_j$ is principal; if $\fa_j = (\alpha)$, then 
      $$ \eta = \frac{2b_j + \sqrt{m}}{\alpha^2} $$
      is a unit with norm $-1$ (equal to $\eps$ if $\alpha$ is chosen suitably). 
\item If $N \eps = + 1$, then there is a subgroup $C$ with index $2$ in the 
      group $\Cl(K)[2]$ of ideal classes of order dividing $2$ such that each 
      class in $C$ is represented by two ramified ideals $\fb_e$ (thus $C$ is 
      the group of strongly ambiguous ideal classes in $K$). Each class in 
      $\Cl(K)[2] \setminus C$ is represented by two ideals $\fa_j$.
\end{itemize}
\end{thm}

The unique pair $(a,b)$ for which $Q_b(x,y) = 1$ is solvable in integers can be
recovered as follows: the extension $\Q(\sqrt{\pm \eps \sqrt{m}\,}\,)$, with
$\eps > 1$ the fundamental unit of $\Q(\sqrt{m}\,)$ and the sign chosen as 
$(-1)^{(m-1)/4} = (\frac 2m)$, is a cyclic quartic extension with 
discriminant $m^3$, hence equal to one of the extensions 
$\Q(\sqrt{m + 2b\sqrt{m}\,}\,)$ for $m = a^2 + 4b^2$, which we 
have studied in \cite{Lqu}.

\subsection*{Remarks}
First traces of Prop. \ref{PrQan} can be found in Euler's work.
In \cite[Prob. 2]{Euler} he proved that if $ap^2 - 1 = q^2$, then
there exist integers $b, c, f, g$ such that 
$$ a = f^2 + g^2, \quad p^2 = b^2 + c^2, \quad
   q = bf+cg \quad \text{and} \quad \pm 1 = bg-cf, $$
and we have $ax^2 + 1 = y^2$ for $x = 2pq$ and $y = 2q^2+1$.

Euler's observation is related to our results as follows: since 
$(b,c,p)$ is a Py\-tha\-go\-rean triple, there exist integers (switch 
$b$ and $c$ if necessary to make $c$ even) $m, n$ such that 
$p = m^2 + n^2$, $b = m^2 - n^2$ and $c = 2mn$; then
$\pm 1 = bg - cf = gm^2 - 2mnf - gn^2$, i.e. $\pm 1$ is represented 
by the binary quadratic form $(g,2f,-g)$.

Euler's results were rediscovered and complemented by Hart \cite{Hart},
A.~G\'erardin \cite{Ger}, Sansone \cite{San,San2}, Epstein \cite{Epst},  
K.~Hardy \& K.~Williams \cite{HaWi} (they first proved the uniqueness 
of the pair $(a,b)$), Arteha \cite{Art}. The connection with Pythagorean 
triples was also observed by Grytczuk, Luca, \& Wojtowicz \cite{GLW}.
 
Generalizations to other Pell equations are due to Sylvester \cite{Syl},
G\"unther \cite{Gu82}, and Bapoungu\'e \cite{Bap,Bap1,Bap2,Bap3,Bap4}.
For details, see \cite[Part II]{LPd}.

\section{Hilbert Class Fields}

Every number field $K$ has a maximal unramified abelian 
extension $L/K$; the field $L$ is called the Hilbert class 
field of $K$. Class field theory predicts that the Galois 
group $\Gal(L/K)$ of this extension is isomorphic to the 
class group $\Cl(K)$. The maximal $2$-extension contained 
in $L/K$ is called the Hilbert $2$-class field, and is 
isomorphic to the $2$-class group $\Cl_2(K)$, the $2$-Sylow 
subgroup of $\Cl(K)$.

In this section we will present a very simple constructing of
unramified cyclic octic extensions of complex quadratic number 
fields. Such constructions are known (see e.g. \cite{CC,Cooke,Kap}, as well
as the classical presentations \cite{Red32,RPell,RR}) for discriminants 
$m = -4p$ for primes $p \equiv 1 \bmod 8$, but even in this case we 
have managed to drastically simplify the construction.

Let $m \equiv 1 \bmod 4$ be a sum of two squares, say $m = a^2 + 4b^2$. 
Let $k = \Q(\sqrt{-m}\,)$ denote the quadratic number field with 
discriminant $\Delta = -4m$. The quadratic extension $k_2 = \Q(i,\sqrt{m}\,)$ 
is unramified and abelian over $\Q$, and is a subfield of the genus 
class field of $k$, which is given by 
$k_\gen = \Q(i,\sqrt{q_1},\ldots, \sqrt{q_t}\,)$, where
$m = q_1 \cdots q_t$ is the prime factorization of $m$.

The quadratic unramified extensions of $k$ have the form 
$K = \Q(\sqrt{c}, \sqrt{-d}\,)$, where $m = cd$ is a 
factorization of $m$ with $c \equiv -d \equiv 1 \bmod 4$.
These factorizations already occurred in (\ref{Edesc1}), where
the factorizations $m = cd$ with $c \equiv 3 \bmod 4$
lead to equations that do not even have solutions modulo $4$.

Unramified cyclic quartic extensions of $k$ containing $k(\sqrt{d_1}\,)$
correspond bijectively to factorizations $\Delta = d_1d_2$ into two 
discriminants $d_1, d_2$ with $(d_1/p_2) = (d_2/p_1) = +1$ for all 
primes $p_1 \mid d_1$ and $p_2 \mid d_2$. Such factorizations are 
called $C_4$-de\-com\-po\-si\-tions of $\Delta = -4m$. In this article 
we deal with unramified cyclic quartic extensions containing 
$k_2 = \Q(i,\sqrt{m}\,)$; such extensions exist if and only if 
$\Delta = -4 \cdot m$ is a $C_4$-decomposition, that is, if and 
only if the primes dividing $m$ are all $\equiv 1 \bmod 8$.

The ambiguous prime ideals $\ne (\sqrt{m}\,)$ all generate ideal 
classes of order $2$; thus their inertia degree in the full 
Hilbert $2$-class field is equal to $2$. This means that if 
$K/k$ is an unramified abelian $2$-extension, then $K$ is the 
full Hilbert $2$-class field if and only if all primes dividing 
$\Delta$ have inertia degree $2$ in $K/k$.

\begin{figure}[ht!]
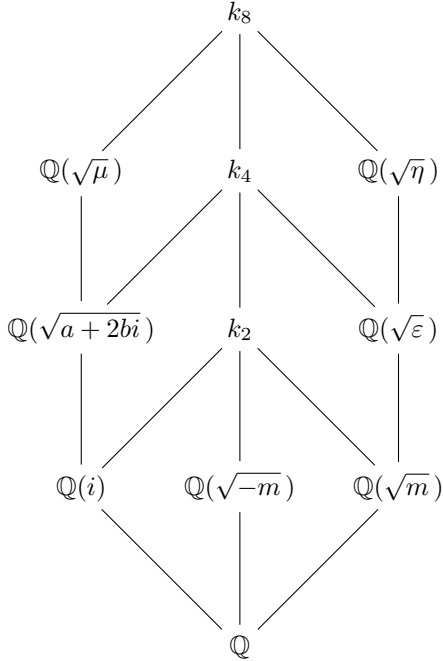

\begin{diagram}
    &         &    k_8   &   & \\          
    & \ruLine &   \uLine         & \luLine &     \\
  \Q(\sqrt{\mu}\,)   &  &  k_4  &  &  \Q(\sqrt{\eta}\,) \\
   \uLine & \ruLine &   \uLine         & \luLine &  \uLine     \\
\Q(\sqrt{a+2bi}\,) &   &  k_2  &  &  \Q(\sqrt{\eps}\,) \\
   \uLine & \ruLine &   \uLine         & \luLine &  \uLine     \\
   \Q(i)   &         &  \Q(\sqrt{-m}\,) &         & \Q(\sqrt{m}\,) \\
        & \luLine &    \uLine        & \ruLine &  \\ 
        &         &      \Q          &         &  
\end{diagram}
\caption{A piece of the $2$-class field tower of $\Q(\sqrt{-m}\,)$.}
\end{figure}

The construction of a cyclic unramified extension $k_4/k$ containing
$k(i) = k(\sqrt{m}\,)$ is classical (see \cite{Herz}): we have to 
solve the equation
\begin{equation}\label{Ekum4}
     A^2 + B^2 - mC^2 = 0.     
\end{equation}
This equation has solutions with $C = 1$, namely $A = a$ and $B = 2b$,
where $m = a^2 + 4b^2$. The extension $k(\sqrt{a+2bi}\,)/k$ is a 
cyclic quartic extension unramified outside of $2$; this extension
is unramified above $2$ (and thus unramified everywhere) if and only 
if $b$ is even, which in turn is equivalent to $m \equiv 1 \bmod 8$.

The Galois group of $k_4/k$ is generated by the element $\sigma$,
whose action on the generators of the extension $k_4/k$ is given 
by the following table:
 $$ \begin{array}{l|rrc}
 \rsp    \alpha         &  i & \sqrt{m}   & \sqrt{a+2bi}          \\ \hline
 \rsp    \sigma(\alpha) & -i & - \sqrt{m} & \sqrt{m}/\sqrt{a+2bi} \\ 
 \rsp  \sigma^2(\alpha) &  i & \sqrt{m}   & -\sqrt{a+2bi}   
     \end{array} $$

We now will construct cyclic octic extensions $k_8/k$ containing $k_4$.
It follows from elementary Galois theory (or see \cite{RPell}) that,  
to this end, we have to solve the diophantine equation 
\begin{equation}\label{EKum8}
   \alpha^2 - (a+2bi) \beta^2 - (a-2bi) \gamma^2 = 0. 
\end{equation}
A solution $(\alpha,\beta,\gamma)$ of this equation is called primitive if 
$\gcd(\alpha,\beta) = 1$.

\begin{thm}\label{TConk8}
Let $m = a^2 + 4b^2$ be a sum of two squares. Then $\sqrt{a+2bi}$
generates a cyclic quartic extension $k_4$ of $k = \Q(\sqrt{-m}\,)$, 
which is unramified if and only if $b$ is even.

Assume that $b$ is even. Then the extension $k_4/k$ can be embedded 
in a cyclic octic extension if and only if Eqn. (\ref{EKum8}) is 
solvable. Let $(\alpha,\beta,\gamma)$ be a primitive solution of 
(\ref{EKum8}); then the extension 
$k_8 = k(\sqrt{\alpha + \beta\sqrt{a+2bi}}\,)$ is a cyclic octic 
extension of $k = \Q(\sqrt{-m}\,)$ unramified outside of $2$ and
containing $k_4$. The extension $k_8/k$ is unramified everywhere 
if and only if $a+2b \equiv \pm 1 \bmod 8$.
\end{thm}

The proof of Thm. \ref{TConk8} is done step by step. 

\subsubsection*{1. The extension $k_8/k$ is cyclic of degree $8$.}

For proving the cyclicity of $k_8/k$ we use a classical result in 
Galois theory (see \cite[Sect. 8.4$^*$]{LL}; each of the following
statements is completely elementary). Let $K/k$ be a normal extension 
with Galois group $G$, and let $L = K(\sqrt{\mu}\,)$ be a quadratic 
extension.
\begin{enumerate}
\item (\cite[F8.9]{LL}) $L/k$ is normal if and only if for every 
      $\sigma \in G$ there is an $\alpha_\sigma \in K$ such that 
      $\mu^{\sigma-1} = \alpha_\sigma^2$.
\item (\cite[F8.10]{LL}) Let $L/k$ be normal. Then we can define an 
      element $[\beta]$ in the second cohomology group $H^2(G,\mu_2)$ 
      with values in $\mu_2 = \{-1,+1\}$ by setting
      $$ \beta(\sigma,\tau) = 
           \alpha_\sigma^\tau \alpha_\tau \alpha_{\sigma\tau}^{-1}. $$
\item (\cite[F8.11]{LL}) If $L/k$ is normal, then $[\beta]$ is the 
      element of the second cohomology group attached to the group extension
      $$ \begin{CD} 
         1 @>>> \mu_2 @>>> \Gal(L/k) @>>> \Gal(K/k) @>>> 1. \end{CD} $$
\item (\cite[Ex. 6, p. 142]{LL}) If $K/k$ is cyclic of even degree, then 
      $L/k$ is cyclic if and only if $\alpha_\sigma^\nu = -1$, where
      $\sigma$ generates $G = \Gal(K/k)$ and where 
      $\nu = \sum_{\tau \in G} \tau$.
\end{enumerate}

Set $\sqrt{a-2bi} = \sqrt{m}/\sqrt{a+2bi}$, 
$\mu = \alpha + \beta\sqrt{a+2bi}$ and $\nu = \alpha + \gamma\sqrt{a-2bi}$. 
We find   
$$ \mu^{\sigma+1} = \alpha \alpha' + \alpha \beta' \sqrt{a-2bi}
                    + \alpha' \beta \sqrt{a+2bi}  + \beta\beta' \sqrt{m}
                 = \mu_\sigma^2 $$
for 
$$ \mu_\sigma = 
   \frac{\alpha + \beta \sqrt{a+2bi} + i\beta' \sqrt{a-2bi}\,}{1+i}, $$
where we have used $\alpha = i \alpha'$ and $\gamma = i\beta'$, as well 
as $(a-2bi)\gamma^2 = \alpha^2 - (a+2bi) \beta^2$. 
With $\alpha_\sigma = \mu_\sigma/\mu$ and $\nu= (1+\sigma)(1+\sigma^2)$
we now find
\begin{align*}
\mu_\sigma^{1+\sigma^2} & = -\beta \beta' \sqrt{m}  & 
       \mu^{1+\sigma^2} & = (a-2bi)\gamma^2  \\
\mu_\sigma^\nu          & = -(\beta\beta')^2 m     &  
\mu^\nu                 & = (\gamma \gamma')^2 m  = -\mu_\sigma^\nu  
\end{align*}
Thus $\alpha_\sigma^\nu = (\mu_\sigma/\mu)^\nu = -1$, which 
proves the claim.

\subsubsection*{2. The extension $k_8/k$ is unramified outside of $2$.}
This follows from a standard argument (see e.g. \cite{RR,LRR1}): if 
$$ \alpha^2- \mu \beta^2 -\mu'\gamma^2 = 0, $$
then 
\begin{equation}\label{Erelmn}
  2(\alpha + \beta \sqrt{\mu})(\alpha + \gamma \sqrt{\mu'})
  = (\alpha + \beta \sqrt{\mu} + \gamma \sqrt{\mu'})^2. 
\end{equation} 

This implies the claim because the ideal generated by 
$\alpha + \beta \sqrt{\mu}$ and its conjugate 
$\alpha + \gamma\sqrt{\mu'}$ is not divisible by any prime
ideal with odd norm.

Equation (\ref{Erelmn}) is a special case of the following 
simple but useful observation:

\begin{lem}
Assume that $Ax^2 - By^2 - Cz^2 = 0$. Then 
\begin{equation}\label{ELeg}
   2(x\sqrt{A} + y\sqrt{B}\,)(x\sqrt{A} + z\sqrt{C}\,)
           = (x\sqrt{A} + y\sqrt{B} + z\sqrt{C}\,)^2.
\end{equation}
\end{lem}

\begin{proof}
This is a straightforward calculation. See also \cite{Lqu}. 
\end{proof}

\begin{cor}
Let $m = a^2 + 4b^2$. Then $\sqrt{a+2bi}$ and $\sqrt{2a+2\sqrt{m}}\,)$
generate the same quadratic extension over $k_2$. If $m = p$ is
prime and $t^2 - pu^2 = -1$, then we also have $k_4 = k_2 (\sqrt{\eps}\,)$
for $\eps = t+u\sqrt{p}$.
\end{cor}

\begin{proof}
From $t^2 - mu^2 = -1$ we get $1 + t^2 - mu^2 = 0$, so (\ref{ELeg})
holds with $A = 1$, $B = -1$ and $C = m$. This shows that 
$$ 2(a+2bi)(a+\sqrt{m}\,) = (a+2bi+\sqrt{m}\,)^2 $$
and proves the first claim. Similarly, $t^2 - mu^2 + 1 = 0$ gives
$$ 2(t+i)(t+u\sqrt{m}\,) = (t+i+u\sqrt{m}\,)^2. $$

Next $t^2 - mu^2 = -1$ implies $t^2 + 1 = mu^2$ and 
$t-i = -i(a+2bi) \rho^2$ as in Subsection \ref{S1.2}. This shows that 
$\sqrt{i(t-i)}$ and $\sqrt{a+2bi}$ generate the same quadratic
extension of $F = \Q(i)$. Observe that since $2i = (1+i)^2$ is a 
square, we also have $F(\sqrt{i(t-i)}\,) = F(\sqrt{2(t-i)}\,)$. 
\end{proof}

\subsubsection*{3. The extension $k_8/k$ is unramified above $(2)$
     if $(\frac2{a+2b}) = 1$.}
We choose the sign of $a$ in such a way that $a \equiv 1 \bmod 4$; 
since $b$ is even, we have $a \equiv 1, 5 \bmod 8$, and the following 
congruences hold modulo $4$:
 \begin{align*}
                    & & a \equiv & 1 \bmod 8   &  &  a \equiv 5 \bmod 8 \\ 
  \Big( \frac{i+\sqrt{a+2bi}}{1+i}\Big)^2 
        & &  \equiv & \ \ \sqrt{a+2bi}  & & \equiv  \ 2 + 2i + \sqrt{a+2bi} \\
  \Big( \frac{2+i+\sqrt{a+2bi}}{1+i}\Big)^2 
        & &  \equiv & \ \ 2 + \sqrt{a+2bi} & &  \equiv  \ 2i + \sqrt{a+2bi} 
 \end{align*}

In Cor. \ref{Corx} below we will show that we can always choose a 
primitive solution $(\alpha,\beta,\gamma)$ of (\ref{EKum8}) in such 
a way that $(2+2i) \mid \alpha$, $\beta \equiv 1 \bmod 2$ and 
$\gamma = i\bbeta \equiv i \bmod 2$.
   
$$ \begin{array}{c|cc}
    x \bmod 2 &  \alpha \bmod 4 & \beta \bmod 4  \\  \hline
        0     &       0         &  \pm 1 + bi \\
        1     &     2 + 2i      &  \pm 1 + (b+2)i 
   \end{array} $$
Thus in the first case we have 
$\alpha + \beta \sqrt{a+2bi} \equiv \pm (1 + bi) \bmod 4$,  
which is a square modulo $4$ if $a \equiv 1 \bmod 8$ and 
$4 \mid b$, i.e., if $a + 2b \equiv 1 \bmod 8$;  moreover we have
$\alpha + \beta \sqrt{a+2bi} \equiv 2+2i \pm (1 + (b+2)i) \bmod 4$,  
which is a square modulo $4$ if $a \equiv 5 \bmod 8$ and 
$b \equiv 2 \bmod 4$, i.e., if $a + 2b \equiv 1 \bmod 8$. 

This completes the proof of Thm. \ref{TConk8}.

\bigskip

A little surprisingly, the diophantine equation (\ref{EKum8}) in 
$\Z[i]$ can be reduced to the following equation in rational integers 
(see \cite[p. 76]{LD}):
\begin{equation}\label{EQb} 
    Q(r,s) = br^2 + ars - bs^2 = 2x^2.
\end{equation} 
In fact, we claim

\begin{prop}
Let $m = a^2 + 4b^2$ for integers $a \equiv 1 \bmod 2$ and $b$.
Then the following statements are equivalent:
\begin{enumerate}
\item[1)] Equation (\ref{EKum8}) is solvable.
\item[2)] $[\frac{a+2bi}{\pi}] = 1$ for all $\pi \mid (a-2bi)$;
\item[3)] The primes $p \mid m$ have inertia degree $1$ in 
      $\Q(i,\sqrt{a+2bi}\,)$ 
\item[4)] $(\frac{2b}p) = +1$ for all primes $p \mid m$.
\item[5)] Equation (\ref{EQb}) is solvable.
\end{enumerate}
\end{prop}

\begin{proof}
For proving that 1) $\Lra$ 2), recall that ternary quadratic equations 
such as (\ref{EKum8}) are solvable globally if and only if they are 
everywhere locally solvable. Since we may disregard one place 
because of the product formula, all we need to check is solvability 
at the primes dividing $a \pm 2bi$. This proves the equivalence 
of 1) and 2); the third statement is a translation of the second 
using the decomposition law in quadratic extensions. 

Next we claim that $[\frac{a+2bi}{\pi}] = (\frac{2b}p)$ for every
prime $\pi \mid (a-2bi)$ with norm $p$; this shows that 2) and 4)
are equivalent. To this end, observe that
$2bi \equiv a \bmod \pi$; this shows that 
$[\frac{a+2bi}{\pi}] = [\frac{4bi}{\pi}] = [\frac{2b}{\pi}] = (\frac{2b}p)$,
where we have used simple properties of residue symbols (see
\cite[Ch. 4]{LRL}) and the fact that $2i = (1+i)^2$ is a square.

Finally we prove that 4) and 5) are equivalent. To this end we observe
that (\ref{EQb}) is solvable in the rationals if and only if it is 
solvable everywhere locally. We may omit proving solvability in $\Q_2$ 
by the product formula. Now $br^2 + ars - bs^2 = 2x^2$ can be written in 
the form $(2br+as)^2 - m s^2 = 8bx^2$; this norm equation is 
solvable if and only if $(\frac{2b}p) = +1$ for all $p \mid m$.
\end{proof}

The fact that (\ref{EKum8}) and (\ref{EQb}) are simultaneously
solvable suggests that there exists an algebraic relation between 
its solutions. Such a relation does indeed exist:

\begin{lem}\label{Lequ}
Let $Q = Q_b = (b,a,-b)$ be a quadratic form with discriminant 
$m = a^2 + 4b^2$. If (\ref{EQb}) has a solution in nonzero integers, then 
\begin{equation}\label{Eabc}
   \alpha = 2x(1+i),  \qquad \beta = r+si, \qquad \gamma = s + ri 
\end{equation}
satisfy Equation (\ref{EKum8}). Moreover, $\gcd(r,s) \mid \gcd(\alpha,\beta)$.
\end{lem}

\begin{proof}
We find
\begin{align*}
       \alpha^2 & = 8x^2i, \\
  (a+2bi) \beta^2  & = (a+2bi)(r^2 - s^2 + 2rsi) \\
                   & = a(r^2-s^2) - 4brs + Q(r,s) \cdot bi \\
                   & =  a(r^2-s^2) - 4brs + 4x^2i, \\
 (a-2bi) \gamma^2  & = -a(r^2-s^2) + 4brs + 4x^2i,
\end{align*}
which immediately implies the first claim.

Clearly the square of $\gcd(r,s)$ divides $2x^2$,
hence $\gcd(r,s) \mid x$. This implies 
$\gcd(r,s) \mid \gcd(\alpha,\beta,\gamma)$ via (\ref{Eabc}).
\end{proof}

Observe that we do not claim that every solution of (\ref{EKum8}) 
comes from a solution of (\ref{EQb}) via the formulas (\ref{Eabc}).

\begin{cor}\label{Corx}
Replacing $b$ by $-b$ if necessary we can always find a solution
$(\alpha,\beta,\gamma)$ of (\ref{EKum8}) such that
$(2+2i) \mid \alpha$, $\beta \equiv 1 \bmod 2$ and $\gamma \equiv i \bmod 2$.
In this case, we have
$\alpha \equiv (2+2i)x \bmod 4$ and $\beta \equiv \pm 1 + (b+2x)i \bmod 4$.
\end{cor}

\begin{proof}
Assume that $br^2 + ars - bs^2 = 2x^2$ has a solution, which we
assume to be primitive ($\gcd(r,s) = 1)$; since $b$ is even and 
$a$ is odd, we must have $2 \mid rs$. If $r$ is even, replacing
$b$ by $-b$ and $(r,s)$ by $(s,-r)$ we get a primitive solution
in which $s$ is even. 

Thus we may assume that $r$ is odd and $s$ is even. Reducing
$br^2 + ars - bs^2 = 2x^2$ modulo $4$ shows that 
$2x^2 \equiv b + s \bmod 4$, which implies $s \equiv b+2x \bmod 4$ 
as claimed. The other claims follow from Lemma \ref{Lequ}.
\end{proof}



We now give some explicit examples.

\subsection*{The Case $m = p$ for primes $p \equiv 1 \bmod 8$.}

If $m = p$ is prime, then  $[\frac{a+2bi}{a-2bi}] = (\frac{2a}p) = 1$ 
if and only if $p \equiv 1 \bmod 8$. Here are a few examples:

$$ \begin{array}{r|rccccc}
   p &  h & a+2bi & \alpha & \beta & \gamma & \mu \\ \hline 
\rsp  17 &  4 & 1 + 4i & 2 + 2i &  1   &  i   & 2+2i + \sqrt{1+4i} \\
\rsp  41 &  8 & 5 + 4i & 2 + 2i &  1   &  i   & 2+2i + \sqrt{5+4i} \\
\rsp  73 &  4 & -3 + 8i & 6 - 6i & 1+2i  & 2-i & 6-6i + (1+2i)\sqrt{-3+8i} \\
\rsp  89 & 12 & 5 + 8i & 10+10i & 3+2i & 2+3i & 10+10i + (3+2i)\sqrt{5+8i} \\
\rsp  97 &  4 & 9 + 4i & 2 + 2i &  1   &  i   & 2+2i + \sqrt{9+4i} 
   \end{array} $$ 

The extension $k_8 = k_4(\sqrt{\mu}\,)$ is unramified outside $2$,
and unramified everywhere if and only if $h \equiv 0 \bmod 8$ (or, 
equivalently, if $a + 2b \equiv \pm 1 \bmod 8$).
The first few examples of unramified extensions are

$$ \begin{array}{r|rccccc}
   p &  h & a+2bi & \alpha & \beta & \gamma & \mu \\ \hline 
\rsp  41 &  8 &   5 +  4i & 2 + 2i &  1   &  i   & 2+2i + \sqrt{5+4i} \\
\rsp 113 &  8 &  -7 +  8i & 2 + 2i & 1+2i & 2+i  & 2+2i + (1+2i)\sqrt{-7+8i} \\
\rsp 137 &  8 & -11 +  4i & 2 + 2i &  1   &  i   & 2+2i + \sqrt{-11+4i} \\
\rsp 257 & 16 &   1 + 16i & 4 + 4i &  1   &  i   & 4+4i + \sqrt{1+16i} 
   \end{array} $$

\subsection*{The Case $m = pq$ for primes $p \equiv q \equiv 1 \bmod 8$.}
In this case, the class of $(2,2,n)$ with $n = \frac{1-pq}2$ is a square. 
Since the forms $(p,0,q)$ are not equivalent to $(2,2,n)$, the class
group contains a subgroup of type $(2,4)$. It contains $(4,4)$
if and only if the class of $(p,0,q)$ is a square, which happens
if and only if $(\frac pq) = 1$.

\subsection*{Case A: $(\frac pq) = -1$}
In this case, the $2$-class group has type $(4,2)$, the square
class being generated by $(2,2,n)$. Since $2$ splits in the three
quadratic extension $k(i)$, $k(\sqrt{p}\,)$ and $k(\sqrt{q}\,)$,
one of them can be embedded into an unramified $C_4$-extension.
By R\'edei-Reichardt, this $C_4$-extension is generated by 
the square root of $a+2bi$, where $m = a^2 + 4b^2$. 

Writing $m = c^2 + 4d^2$ as a sum of squares in an essentially different
way does not produce anything new because of
$$ k_2(\sqrt{a+2bi}\,) = k_2(\sqrt{p(a+2bi)}\,) = k_2(\sqrt{c+2di}\,). $$
Since $(\frac pq) = -1$, this implies that exactly one among the two 
equations of type (\ref{E2Dn}) has a solution.

The unramified extension $k_2(\sqrt{p},\sqrt{a+2bi}\,)$ of type
$(2,4)$ over $k$ is the full Hilbert $2$-class field if and
only if $(\frac{2}{a+2b}) = -1$. Observe that 
$(\frac{2}{a+2b})= (\frac{2}{c+2d})$ since $2$ splits in both or
in neither of the two quartic extensions.

$$ \begin{array}{rr|c|c|c|cl}
    p &  q  & \Cl(-4pq) &  \text{generators} 
      & (b,a,-b) & (\frac{2}{a+2b}) & \ (r,\ s,\ x) \\ \hline 
   17 &  41 & (4, 2)    & (19,10,38) & ( 8, 21, -8)  & -1 & (1,0,2) \\
      &     &           & (29,24,29) & (12, 11, -12) & -1 & - \\
   17 &  73 & (16, 2)   & (3,2,414)  & ( 2, 35, -2)  & +1 & (1,0,1)  \\
      &     &           & (17,0,73)  & (10, 29, -10) & +1 & - \\
   17 &  97 & (24, 2)   & (3,2,550)  & (16, 25, -16) & +1 & (1,2,1)  \\
      &     &           & (17,0,97)  & (20,  7, -20) & +1 & - \\
   17 & 113 & (20, 2)   & (35,-4,55) & (18, 25, -18) & -1 & (5,-2,8) \\
      &     &           & (17,0,113) & (10, 39, -10) & -1 & - \\
   41 &  89 & (20, 2)   & (31,6,118) & (10, 57, -10) & -1 & (6,-1,2)  \\
      &     &           & (65,48,65) & (30,  7, -30) & -1 & - \\
   41 &  97 & (28, 2)   & (3,2,1326) & ( 8, 61, -8)  & -1 & (1,0,2) \\
      &     &           & (41,0,97)  & (28, 29, -28) & -1 & - 
  \end{array} $$

\subsection*{Case B: $(\frac pq) = +1$}

Let $m = pq = a^2 + 4b^2$. Writing $\pi = a_1 + 2b_1i$ and 
$\rho = a_2 + 2b_2i$ we find 
$\pi \rho  = a_1a_2 - 4b_1b_2 + 2(a_1b_2 + a_2b_1)i$ and
$\pi \brho = a_1a_2 + 4b_1b_2 + 2(a_1b_2 - a_2b_1)i$.
Since $[\frac{\bpi}{\pi}] = 1$ etc., we find 
$[\frac{\pi\rho}{\bpi}] = [\frac{\rho}{\bpi}] = [\frac{\rho}{\pi}] = 
 (\frac pq)_4^{\phantom{p}}(\frac qp)_4^{\phantom{p}}$. This shows
that the solvability conditions $[\frac{a+2bi}{\pi}] = +1$ are
equivalent to $(\frac pq)_4^{\phantom{p}}(\frac qp)_4^{\phantom{p}} = 1$.

Note that $(\frac2{a+2b}) = [\frac{1+i}{a+2bi}] = 
  (\frac2{pq})_4^{\phantom{p}}(\frac{pq}2)_4^{\phantom{p}}$,
where $(\frac m2)_4^{\phantom{p}} = (-1)^{(m-1)/8}$ for integers
$m \equiv 1 \bmod 8$.

\begin{thm}\label{T14}
Assume that $m = pq = a^2 + 4b^2$ is the product of two primes 
$p \equiv q \equiv 1 \bmod 8$ with $(\frac pq) = 1$. Then the
quartic extension $k(\sqrt{a+2bi}\,)/k$ can be embedded into an 
unramified cyclic octic extension if and only if 
 $(\frac2{pq})_4^{\phantom{p}}(\frac{pq}2)_4^{\phantom{p}} = 
  (\frac pq)_4^{\phantom{p}}(\frac qp)_4^{\phantom{p}} = 1$.
More precisely, the following statements are true:
\begin{enumerate}
\item The primes above $2$ split in the quartic extension 
      $k(\sqrt{a+2bi}\,)/k$ if and only if 
      $(\frac2{pq})_4^{\phantom{p}}(\frac{pq}2)_4^{\phantom{p}} = 1$.
\item Equation (\ref{EKum8}) is solvable if and only if
      $(\frac pq)_4^{\phantom{p}}(\frac qp)_4^{\phantom{p}} = 1$.
\item The equation $Q(r,s) = 2x^2$, where $Q = (b,a,-b)$, is solvable
      in integers if and only if  
      $(\frac pq)_4^{\phantom{p}}(\frac qp)_4^{\phantom{p}} = 1$.
\item The cyclic octic extension $k_8/k$ constructed from a solution 
      of (\ref{EKum8}) is unramified above the primes dividing $(2)$
      if and only if 
      $(\frac2{pq})_4^{\phantom{p}}(\frac{pq}2)_4^{\phantom{p}} = 1$.
\end{enumerate}
\end{thm}

\begin{proof}
The condition $(\frac bp) = 1$ is equivalent to 
$(\frac{a_1b_2 + a_2b_1}{p}) = 1$, which by Burde's rational
reciprocity law is equivalent to 
$(\frac pq)_4^{\phantom{p}}(\frac qp)_4^{\phantom{p}} = 1$.
\end{proof}

In the table below, $\eps$ denotes the values of $(\frac pq)_4^{\phantom{p}}$
and $(\frac qp)_4^{\phantom{p}}$, respectively.

$$ \begin{array}{rr|c|c|c|ccl}
    p &  q  & \Cl(-4pq) &  \text{generators} 
      & (b,a,-b) & (\frac{2}{a+2b}) & \eps & \ (r,\ s,\ x) \\ \hline 
   17 &  89 & (4, 4)    & (11,8,139) & (14, 27, -14) & +1 & -1 & - \\
      &     &           & (19,16,83) & ( 6, 37, -6)  & +1 & +1 & - \\
   17 & 137 & (8, 4)    & (5,2,466)  & (24,  5, -24) & -1 & -1 & - \\
      &     &           & (35,-8,67) & (20, 27, -20) & -1 & +1 & - \\
   41 &  73 & (12, 4)   & (3,2,998)  & (26, 17, -26) & -1 & +1 & - \\
      &     &           & (29,-18,106) & (14, 47, -14) & -1 & -1 & - \\
   41 & 113 & (8, 4)    & ( 7, 2, 662)   & (34,  3, -34) & +1 & -1 & - \\
      &     &           & (23, -12, 203) &  ( 6, 67, -6)  & +1 & +1 & - \\
   73 &  89 & (8, 8)    & (19, 2,342)  &  (32, 49, -32) & +1 & +1 & (1,0,4) \\
      &     &           & (26,-18,253) &  ( 8, 79, -8)  & +1 & +1 & (1,0,2) 
  \end{array} $$

The cyclic quartic extensions of $\Q(\sqrt{-17 \cdot 137}\,)$ are
generated by 
$$ \sqrt{5+48i}, \quad \sqrt{-51+4\sqrt{17}}, \quad 
   \sqrt{-3699 + 316\sqrt{137}} $$ 
Since $(1+i)$ is inert in $\Q(\sqrt{5+48i}\,)/\Q(i)$, this extension
cannot be embedded into a cyclic quartic unramified extension by Thm.
\ref{T14}.(1). Similarly, $1+4i$ is inert in this extension since 
$[\frac{11+4i}{1+4i}] = (\frac{10}{17}) = -1$.

\bigskip

Let us now make a few simple remarks on special cases where the
equation
\begin{equation}\label{Ers}
       br^2 + ars - bs^2 = 2x^2
\end{equation}       
is solvable.
\begin{enumerate}
\item $b = 2d^2$: then $(r,s,x) = (b,a,bd)$ is a solution. Note that
      $p = a^2 + 4b^2 = a^2 + (2d)^4$ in this case.
\item $b = d^2$: then $(r,s,x) = (2b,a+e,f)$ is a solution, where 
      $p = e^2 + 2f^2$.  
\end{enumerate}

\section{$2$-Descent on Elliptic Curves}

In this section we will give an exposition of Ap\'ery's lecture
\cite{Apery}. The forms $Q_b$ also show up in several other investigations
of elliptic curves with primes $p = a^2 + 4b^2$ as parameters (see e.g. 
\cite{ST}), and we have selected Ap\'ery's article mainly because his 
presentation was the least polished. Ap\'ery starts by recalling a 
conjecture made by Mordell in Debrecen 1968: for each prime 
$p \equiv 5 \bmod 8$, the curve
\begin{equation}\label{Ap1} 
      y^2 = px^4 + 1
\end{equation}
has a nontrivial (i.e., $(x,y) \ne (0, \pm 1)$) rational point 
(this is also predicted by the more general parity conjecture).
Since (\ref{Ap1}) is a curve of genus $1$ with a rational point,
it is an elliptic curve. The conjecture that (\ref{Ap1}) has 
nontrivial rational points is equivalent to the conjecture that
the elliptic curve has Mordell-Weil rank $1$.

Consider more generally an elliptic curve 
$$ E: y^2 = x(x^2 + Ax + B) $$ 
defined over $\Q$ with a rational torsion point $(0,0)$ of order $2$.
Each rational affine point on $E$ has the form $(x,y)$ with 
$$ x = b_1 \frac{m}{e^2}, \quad y = b_1 \frac{mn}{e^3}, $$
and comes from a rational point on one of the torsors
$$ T: n^2 = b_1m^4 +am^2e^2 + b_2e^4 $$
with $b_1b_2 = B$ and $b_1$ squarefree.
 
The curve 
$$ E': y^2 = x(x^2 + A'x + B'), \quad  \text{with} \quad
        A' =  -2Ax \ \text{and} \ B' = A^2-4B, $$ 
is $2$-isogenous to $E$. 

Specializing to $A = 0$ and $B = p$ we find that the elliptic curves
\begin{equation}\label{Eell}
   E: y^2 = x(x^2 + p)  \quad \text{and} \quad E': y^2 = x(x^2 - 4p)
\end{equation}
have the torsors
$$  n^2 = pm^4 + e^4 \quad \text{and} \quad 
    n^2 = m^4 - 4pe^4, \quad  n^2 = pm^4 - 4e^4. $$ 
By the theory of $2$-descent on elliptic curves, (\ref{Eell})
has rank $r \le 1$, with equality if and only if the torsor
$n^2 = pm^4 - 4e^4$ has a nontrivial rational point.

Now consider the torsor
$$  Z^2 = pX^4 - 4Y^4 $$
(we are using Ap\'ery's notation) and write $p = a^2 + 4b^2$.
From 
$$ pX^4 = Z^2 + 4Y^4 = (Z + 2iY^2)(Z - 2iY^2) $$
we get, using unique factorization in $\Z[i]$,  
$$ Z + 2iY^2 = (a+2bi)(\xi + i \eta)^4, $$
where $X = \xi^2 + \eta^2$. Comparing real and imaginary parts yields
\begin{equation}\label{EYxi}
 Y^2 = b(\xi^4 - 6\xi^2\eta^2 + \eta^4) + 2a\xi\eta(\xi^2 - \eta^2). 
\end{equation}
Setting $r = \xi^2 - \eta^2$ and $s = 2\xi\eta$, we find
\begin{equation}\label{EAp}
   X^2 = r^2 + s^2, \qquad Y^2 = br^2 + ars - bs^2. 
\end{equation}
Given the last pair of equations we parametrize the Pythagoren 
equation $X^2 = r^2 + s^2$ via $r = \xi^2 - \eta^2$, $s = 2\xi\eta$
and $X = \xi^2 - \eta^2$, plug the results into the second equation
and retrieve (\ref{EYxi}). 

Thus finding a rational point on $E$ boils down to finding a 
simultanous representation of squares for the pair of forms
$Q = (1,0,1)$ and $Q_b(b,a,-b)$.

\bigskip\noindent{\bf Example.}
For $p = 797$, we have $a = 11$ and $b = 13$. We find
\begin{align*}
  \xi & = &    1462  & &  \eta & = &     771 \\
   X  & = & 2731885  & &   Y   & = & 1773371 \\
   x  & = & \frac{5948166935620325}{3144844703641} & &   
   y  & = & \frac{458544116976814482315845}{5576976396940543811}
\end{align*}
since $x = p\frac{X^2}{Y^2}$ and $y = p\frac{XZ}{Y^3} $, where
$Z = 210600981540301$.

\section{Cyclic Quartic Fields}

In the late 1940s, Hasse was interested in the explicit arithmetic
of abelian extensions; in 1948 he presented a memoir \cite{HasEK} 
on the computation of unit groups and class numbers of cyclic cubic 
and quartic fields, and in 1952 he published his book \cite{HasAZ} 
on the investigation of class numbers of abelian number fields.

The quadratic form $Q_b$ shows up in Hasse's treatment of cyclic
quartic number fields in \cite{HasEK}. Since this work has remained
largely obscure we would like to provide as much background as is
necessary to begin to appreciate Hasse's results.

Let $K/\Q$ be a cyclic quartic extension; let $F$ denote its
conductor, and $H$ the subgroup of the group $D$ of nonzero ideals 
in $\Q$ coprime to $F$ that corresponds to this extension by 
class field theory. Thus $D/H \simeq \Gal(K/\Q)$; let $\chi$ be the
ray class character on $D/H$, and $T = -\sum \chi(t)^{-1} e(t)$
the corresponding Gauss sum. 

Let $k$ denote the quadratic subfield of $K$; its conductor $f$
divides $F$, hence we can write $F = fG$ for some integer $G$.
Hasse proves that there exist integers $a, b$  such that 
$f = a^2 + 4b^2$ and 
$$ K = \Q \bigg(\sqrt{\chi(-1) G \frac{f + a\sqrt{f}}2}\,\bigg). $$

His first main result is a description of an integral basis of $K$
in terms of invariants of the field, that is, in terms of Gauss sums.
If $\tau$ denotes the Gauss sum attached to $\chi^2$, then the
algebraic integers in $k$ have the form $\frac12(x + y \tau)$
with $x \equiv \tau y \bmod 2$. Hasse succeded in determining the
ring of integers in $K$ in a similar way. In fact, for elements
$x \in k$ and $y \in \Q(i)$ he observes that every element of $K$
can be represented in the form 
$$ [x,y] = \frac12 \bigg( x + \frac{y\,T + \by\, \tT}2 \bigg), $$
where $\by$ denotes the complex conjugate of $y$ and where 
$\tT = \chi(-1) \bT$.

The ring of integers in $K$ consists of all elements $[x,y]$
with $x \in \cO_k$ and $y \in \Z[i]$ such that 
$$ x \equiv F \cdot 
     \frac{\Tr(\frac{1+i}2 y) + \Tr(\frac{1-i}2 y)\,\tau}2 \bmod 2, $$
where $\Tr$ denotes the trace of $\Q(i)/\Q$.     

Now Hasse observes that the product $x(y\,T + \by\, \tT)$ can be written
in the form
$$ x(y\,T + \by\, \tT) = (x \circ y)\,T + \ov{x \circ y}\, \tT, $$
where 
$$ x \circ y = \frac{x_0y + x_1(a-2bi) \by}2, \qquad
\text{with $x = \frac12(x_0 + x_1 \tau)$ and $y \in \Z[i]$}. $$
This defines an action of $k^\times$ on $\Q(i)^\times$ (which induces an 
action of $\cO_k$ on $\Z[i]$); in fact, the operator product  $x \circ y$
has the following formal properties:

\begin{lem}
For all $x \in k^\times$ and all $y \in \Q(i)$ we have
\begin{enumerate}
\item $1 \circ y = y$.
\item $(x_1 x_2) \circ y = x_1 \circ (x_2 \circ y)$.
\item $x \circ y$ is $\Q$-bilinear: $qx \circ y = x \circ qy = q(x \circ y)$
      for all $q \in \Q$. 
\item $x \circ iy = i(x' \circ y)$, where $x'$ is the conjugate of $x$.
\item $x \circ y = 0$ if and only if $x = 0$ or $y = 0$.
\end{enumerate}
\end{lem}
Hasse uses this action for simplifying the numerical computation 
(he was working with pencil and paper!) of products in the number 
field $K$. The computation of squares, for examples, is
facilitated by the observation that 
$$ [x,y]^2 = \bigg[\frac12\Big(x^2 + \chi(-1)G 
            \frac{N(y)f + \phi(y) \tau}{2}\Big), x \circ y \bigg], $$
where $N$ and $S$ denote the trace in $\Q(i)/\Q$ and where $\phi$ is 
the binary quadratic form defined for $y = r + si$ by
$$ \phi(y) = \frac12 S((a+2bi)y^2) = ar^2 - 4brs - as^2. $$
In the expressions giving the action of the Galois group on $[x,y]$, the form 
$$ \hphi(y) = - \frac14 S(i(a+2bi)y^2) = br^2 + ars - bs^2 $$
shows up, for which Hasse observes the identity (see \cite[(19)]{HasEK}
$$ \hphi(\alpha \circ y) = N(\alpha) \hphi(y) $$
for $\alpha \in \cO_k$ and $y \in \Z[i]$.

Hasse's goal was characterizing the unit groups of cyclic quartic
extensions in a way similar to the real quadratic case, where the
fundamental unit is uniquely determined by the minimal solution of 
the Pell equation $T^2 - mU^2 = 4$.

In the real cyclic case $K/\Q$, let $k$ denote the quadratic
subfield of $K$. The unit group $E_K$ of $K$ is described by 
the following invariants:
\begin{enumerate}
\item the unit group $E_k$ of the quadratic subfield,
\item the group $E_{K/k}$ of relative units satisfying $N_{K/k}\eta = \pm 1$,
      and 
\item the unit index $Q = (E_K:E_{K/k}E_k)$.
\end{enumerate}
These invariants are characterized as follows:
\begin{enumerate}
\item The unit index $Q$ is either $1$ or $2$;
\item There is a unit $\eta$ such that the group $E_{K/k}$ is generated 
      by $-1$, $\eta$ and its conjugate $\eta'$.
\end{enumerate}
This unit $\eta$ can be chosen in an essentially unique way among its
conjugates etc., and this unit is then called {\em the} relative 
fundamental unit of $K$. Hasse's main result is

\begin{thm}
The relative fundamental unit $\eta$ of $K$ is the relative unit 
$\eps \ne \pm 1$ with the property that $|S_{K/\Q}(\eps^2)|$ is minimal.
\end{thm}




\begin{thebibliography}{99}

\bibitem{Apery} R. Ap\'ery,
{\em Points rationels sur certaines courbes alg\'ebriques},
Journ\'ees Arithm\'etiques de Bordeaux 1974, 
Ast\'erisque 24 - 25 (1975), 229--235
%

\bibitem{Art} S.N. Arteha,
{\em Method of hidden parameters and Pell's equation},
JPJ Algebra Number Theory Appl. {\bf 2} (2002), 21--46
%

\bibitem{Bap} L. Bapoungu\'e, 
{\em Sur la r\'esolubilit\'e de l'\'equation
     $ax^2 + 2bxy - kay^2 = \pm 1$}, 
Th\`ese Univ. Caen, 1989; see also
C. R. Acad. Sci. Paris {\bf 309} (1989), 235--238
%

\bibitem{Bap1} L. Bapoungu\'e, 
{\em Un crit\`ere de r\'esolution pour l'\'equation diophantienne
     $ax^2 + 2bxy - kay^2 = \pm 1$},
Expos. Math. {\bf 16} (1998), 249--262
%

\bibitem{Bap2} L. Bapoungu\'e, 
{\em Sur la r\'esolubilit\'e de l'\'equation
     $ax^2 + 2bxy - 8ay^2 = \pm 1$}, 
IMHOTEP, J. Afr. Math. Pures Appl. {\bf 3} (2000), 97--111;
%

\bibitem{Bap3} L. Bapoungu\'e, 
{\em Sur les solutions g\'enerales de l'\'equation
     diophantienne $ax^2 + 2bxy - kay^2 = \pm 1$}, 
Expos. Math. {\bf 18} (2000), 165--175
%

\bibitem{Bap4} L. Bapoungu\'e, 
{\em The diophantine equation $ax^2+2bxy - 4ay^ 2 = \pm 1$},
Intern. J. Math. Math. Sci. {\bf 35} (2003), 2241--2253
%

\bibitem{BH} A. Barker-Hoyt,
{\em Gauss composition and Bhargava's cube law},
master thesis Univ. Maine, 2005
%

\bibitem{Bha} M. Bhargava,
{\em Gauss composition and generalizations},
Algorithmic number theory (Sydney, 2002),  1--8, 
Lecture Notes in Comput. Sci., 2369, Springer, Berlin, 2002
%

\bibitem{Brw1} E. Brown,
{\em Representation of discriminantal divisors by binary quadratic forms},
J. Number Theory {\bf 3} (1971), 213--225
%

\bibitem{Brw2} E. Brown,
{\em Discriminantal divisors and binary quadratic forms},
%

\bibitem{CC} H. Cohn, G. Cooke,
{\em Parametric form of an eight class field},
Acta Arith. {\bf 30} (1976), 367--377
%

\bibitem{Cooke} G. Cooke,
{\em Construction of Hilbert class field extension of 
     $K = \Q(\sqrt{-41}\,)$},
unpublished lecture notes Cornell Univ., 1974
%

\bibitem{Cox} D.A. Cox,
{\em Primes of the form $x^2 + ny^2$. Fermat, Class Field Theory,
     and Complex Multiplication}, John Wiley 1989
%

\bibitem{Epst} P. Epstein,
{\em Zur Aufl\"osbarkeit der Gleichung $x^2-Dy^2=-1$},
J. Reine Angew. Math. {\bf 171} (1934), 243--252
%

\bibitem{Euler} L. Euler,
{\em Nova subsidia pro resolutione formulae $axx + 1 = yy$},
Opusc. anal. {\bf 1} (1783), 310; Comm. Arith. Coll. {\bf II}, 35--43;
Opera Omnia I-4, 91--104
%

\bibitem{Fla} D. Flath,
{\em Introduction to number theory},
Wiley \& Sons, New York, 1989
%

\bibitem{Ger} A. G\'erardin,
{\em Sur l'\`equation $x^2 - Ay^2 = 1$}, 
L'Ens. math. {\bf 19} (1917), 316--318;
Sphinx-\OE dipe {\bf 12} June 15, 1917, 1--3
%


\bibitem{GLW} A. Grytczuk, F. Luca, M. Wojtowicz,
{\em The negative Pell equation and Pythagorean triples},
Proc. Japan Acad. {\bf 76} (2000), 91--94
%

\bibitem{Gu82} S. G\"unther, 
{\em Ueber einen Specialfall der Pell'schen Gleichung},
Bl\"atter f\"ur das Bayerische Gymnasial- und Realschulwesen
{\bf 17} (1882), 19--24
%

\bibitem{HaWi} K. Hardy, K. Williams,
{\em On the solvability of the diophantine equation 
     $dV^2 - 2e VW - dW^2 = 1$},
Pac. J. Math. {\bf 124} (1986), 145--158
%

\bibitem{Hart} D.S. Hart, 
{\em Solution of an indeterminate problem},
Analyst {\bf 5} (1878), 118--119
%

\bibitem{HasEK} H. Hasse,
{\em Arithmetische Bestimmung von Grundeinheit und Klassenzahl in 
     zyklischen kubischen und biquadratischen Zahlk\"orpern},
Abh. Deutsch. Akad. Wiss. Berlin (1950), 3--95; Math. Abh. III, 289--379
%

\bibitem{HasAZ} H. Hasse,
{\em \"Uber die Klassenzahl abelscher Zahlk\"orper},
Akademie Verlag Berlin, 1952
%

\bibitem{Herz} C. Herz,
in {\em Seminar on complex multiplication}, Springer
%

\bibitem{Kap} P. Kaplan,
{\em Unit\'es de norme $-1$ de $\Q(\sqrt{p}\,)$ et corps de classes
     de degr\'e $8$ de  $\Q(\sqrt{-p}\,)$ o\`u $p$ est un nombre 
     premier congru \`a $1$ modulo $8$},
Acta Arith. {\bf 32} (1977), 239--243
%

\bibitem{LD} F. Lemmermeyer,
{\em Die Konstruktion von Klassenk\"orpern},
Ph. D. thesis 1994, Univ. Heidelberg
%

\bibitem{LRR1} F. Lemmermeyer,
{\em Rational Quartic Reciprocity},
Acta Arithmetica {\bf 67} (1994), 387--390
%

\bibitem{LRL} F. Lemmermeyer,
{\em Reciprocity Laws. From Euler to Eisenstein},
Springer-Verlag 2000
%

\bibitem{LPd} F. Lemmermeyer,
{\em Higher descent on Pell conics}, 
  I. {\em From Legendre to Selmer}, arXiv: 0311309v1;  
 II. {\em Two centuries of missed opportunities}, arXiv: 0311296v1;
III. {\em The first $2$-descent}, arXiv: 0311310v1
%

\bibitem{LPep} F. Lemmermeyer,
{\em Binary Quadratic Forms and Counterexamples to 
     Hasse's Local-Global Principle},
preprint 2010
%

\bibitem{Lqu} F. Lemmermeyer,
{\em Relations in the $2$-Class Group of Quadratic Number Fields},
J. Austr. Math. Soc. (2012), to appear
%

\bibitem{LL} F. Lorenz, F. Lemmermeyer,
{\em Algebra I}, Elsevier/Springer-Verlag 2007
%

\bibitem{Pall} G. Pall,
{\em Discriminantal divisors of binary quadratic forms}, 
J. Number Theory {\bf 1} (1969), 525--532
%

\bibitem{Red32} L. R\'edei,
{\em  Die Anzahl der durch $4$ teilbaren Invarianten der
   Klassengruppe eines beliebigen quadratischen Zahlk\"{o}rpers},
 Math. Naturwiss. Anz. Ungar. Akad. d. Wiss. {\bf 49}
   (1932),  338--363 
%

\bibitem{RPell} L. R\'edei,
{\em Die $2$-Ringklassengruppe des quadratischen Zahlk\"orpers
     und die Theorie der Pellschen Gleichung},
Acta Math. Acad. Sci. Hungaricae {\bf 4} (1953), 31--87
%

\bibitem{RR} L.~R\'edei, H. Reichardt,
{\em  Die Anzahl der durch $4$ teilbaren Invarianten der Klassengruppe
        eines beliebigen quadratischen Zahlk\"{o}rpers},
 J. Reine Angew. Math. {\bf 170} (1933),  69--74 
%

\bibitem{San} G.~Sansone,
{\em Sulle equazioni indeterminate delle unit\`a di norma negativa
    dei corpi quadratici reali}, Rend. Acad. d. L. Roma 
(6) {\bf 2} (1925), 479--484 
%

\bibitem{San2}  G.~Sansone,
{\em Ancora sulle equazioni indeterminate delle unit\`a di norma 
  negativa dei corpi quadratici reali}, Rend. Acad. d. L. Roma 
(6) {\bf 2} (1925), 548--554
%

%

\bibitem{ST} R.J. Stroeker, J. Top,
{\em On the equation $Y^2 = (X+p)(X^2+p^2)$},
Rocky Mt. J. Math. {\bf 24} (1994), 1135--1161 
%

\bibitem{Syl} J.J. Sylvester,
{\em Mathematical Question 6243},
Educational Times {\bf 34} (1881), 21--22
%

\end{thebibliography}
\end{document}